\documentclass{amsart}
 \newtheorem{thm}{Theorem}
 \newtheorem{cor}{Corollary}

 \theoremstyle{definition}
 
 \theoremstyle{remark}

 \numberwithin{equation}{section}

\begin{document}
%
%
%
%
%
%
%
%
%
\title
 {An inequality for means with applications}
\author[J.-C. Schlage-Puchta]{Jan-Christoph Schlage-Puchta}

\address{
Mathematisches Institut\\
Eckerstr. 1\\
79104 Freiburg\\
Germany
}

\email{jcp@math.uni-freiburg.de}



\date{Novembre 25, 2006}

\begin{abstract}
We show that an almost trivial inequality between the first and second
moment and the maximal value of a random variable can be used to
slightly improve deep theorems.
\end{abstract}

\maketitle

One often estimates individual values of a function by computing a
certain means. This approach is particularly useful in situations where
the conjectured maximum of a function is close to its mean, as is
often the case in number theory. Here, we give a simple method which
sometimes allows to improve the resulting estimates. To ease
aplications, we formulate our result in the language of probability
theory.
\begin{thm}
Let $\xi$ be a non-negative real random variable, and suppose that
$\mathbf{E}\xi=1$, 
and $\mathbf{E}\xi^2=a$ with $a>1$. Then the probability $P(\xi\geq a)$
is positive, and for every $b<a$ we have 
\[
\int\limits_{|\xi|>b} \xi^2 \geq a-b.
\]
\end{thm}
\begin{proof}
We have
\[
\int\limits_{|\xi|\leq b} \xi^2 \leq b\int\limits_{|\xi|>b} \xi \leq b\int\xi = b,
\]
which implies the claimed inequality. Now suppose that 
$|\xi|<a$ almost surely. Then there exists some $\epsilon>0$, such that
$\int\limits_{|\xi|\leq a-\epsilon} \xi\geq\frac{1}{2}$, and we obtain
\begin{multline*}
a = \int \xi^2 = \int\limits_{|\xi|\leq a-\epsilon} \xi^2 +
\int\limits_{|\xi|> a-\epsilon} \xi^2\\
 \leq (a-\epsilon)\int\limits_{|\xi|\leq a-\epsilon} \xi +
 a\int\limits_{|\xi|>a-\epsilon} \xi \leq \frac{(a-\epsilon) + a}{2} < a, 
\end{multline*}
a contradiction.
\end{proof}
We now give three applications to quite different areas.

Our first application shows that the fourth moment of the Riemann
$\zeta$-function is dominated by large values of $\zeta$, in fact, by values
which are so large that the fourth moment itself cannot guarantee them
to exist.

\begin{cor}
There is a constant $C$ such that for $t>T_0(\epsilon)$ and $H>T^{2/3}\log^CT$
we have
\[
\int\limits_{\{t\in[T, T+H]: |\zeta(\frac{1}{2}+it)|>\frac{1}{4\pi^2}\log^{3/2}
  t\}} |\zeta(\frac{1}{2}+it)|^4\;dt \geq \frac{1-\epsilon}{4\pi^2}T\log^4T +
\mathcal{O}(T\log^3 T).
\]
\end{cor}
\begin{proof}
Ingham proved 
\[
\int_0^T |\zeta(\frac{1}{2}+it)|^2\;dt = T\log\frac{T}{2\pi}+(2\gamma-1)T+\mathcal{O}/T^{1/2+\epsilon}),
\]
and Ivic and Motohashi\cite{IM} showed that 
\[
\int_0^T |\zeta(\frac{1}{2}+it)|^4\;dt = T\cdot P(\log T) + \mathcal{O}(T^{2/3}\log^CT),
\]
where $P$ is a polynomial of degree 4 with leading term
$\frac{1}{2\pi^2}$ (confer also \cite{HB}).
Now apply Theorem 1 by setting
\[
\xi=H\cdot|\zeta(\frac{1}{2}+it)|^2\Bigg(\int\limits_T^{T+H}
|\zeta(\frac{1}{2}+it)|^2\;dt\Bigg)^{-1}
\]
for $t\in[T, T+H]$ chosen at random and $b=\frac{1}{4\pi^2}\log^2 T$.
\end{proof}

Our second application slightly improves on the approach of Szekeres
and Tur\'an \cite{ST} on the problem of Hadamard-matrices.

\begin{cor}
For every $\epsilon>0$ and $n>n_0(\epsilon)$ there exists a skew-symmetric
$n\times n$-matrix $A$ with 
entries $\pm 1$ satisfying
\[
|\det A|> \left(\frac{n}{64 \pi e^5}\right)^{1/4}e^{\sqrt{n}} \sqrt{n!}. 
\]
\end{cor}
\begin{proof}
Let $A$ be a random skew-symmetric $n\times n$-matrix with entries $\pm
1$, and let $s_k(n)$ be the $k$-th mean of the determinant of
$A$. Szekeres\cite{Szekeres} showed that 
\begin{eqnarray*}
s_1(n) & \sim & \frac{1}{\sqrt[4]{8\pi en}} e^{\sqrt{n}} \sqrt{n!},\\
s_2(n) & \sim & \frac{1}{\sqrt{32\pi e^3}} e^{2\sqrt{n}} \sqrt{n!}.
\end{eqnarray*}
Our claim now follows by applying our theorem to $\frac{\det A}{s_1(n)}$.
\end{proof}

Our last result improves on the work of Kerov and Vershik \cite{KV}
concerning the largest degree of an irreducible character of the
symmetric group. 

\begin{cor}
Let $\epsilon>0$ be given. Then for every $n>n_0(\epsilon)$ there exists an
irreducible character $\chi$ of $S_n$ with $\chi(1)>(1-\epsilon)e^{1/4}\sqrt{\pi
  n}e^{-\sqrt{n}}\sqrt{n!}$.
\end{cor}
\begin{proof}
All irreducible complex representations of $S_n$ can be realized over
$\mathbb{R}$, thus
\[
\sum_{\chi} \chi(1) = \#\{\pi\in S_n: \pi^2=\mathrm{id}\} \sim
\frac{e^{\sqrt{n}-\frac{1}{4}} }{2\sqrt{\pi n}}\sqrt{n!},
\]
whereas the orthogonality relation implies $\sum_\chi\chi(1)^2=n!$. Finally, the
number of irreducible characters equals the number $p(n)$ of
partitions of $n$, for which we have the asymptotic formula
\[
p(n)\sim \frac{1}{4n\sqrt{3}}e^{\pi\sqrt{2n/3}}.
\]
Define a random variable $\xi$ as $\frac{\chi(1)}{\sqrt{n!}}$, where $\chi$ is
chosen at random 
among all irreducible characters, where each character has the same
probability. Then we obtain
\begin{eqnarray*}
\mathbf{E}\xi & \sim & \frac{2\sqrt{3n}}{e^{1/4}\sqrt{\pi}}\exp\big((1-\pi\sqrt{2/3})\sqrt{n}\big)\\
\mathbf{E}\xi^2 & \sim & 4n\sqrt{3} \exp\big(-\pi\sqrt{2n/3}\big),
\end{eqnarray*}
and our claim follows.
\end{proof}

\end{document}